\documentclass{amsart}
\usepackage{amsmath,amssymb,amsthm}
\usepackage{color}
\usepackage{enumerate}
\usepackage{cleveref}

\newtheorem{thm}{Theorem}[section]
\newtheorem{cor}[thm]{Corollary}

\newtheorem{lem}[thm]{Lemma}
\newtheorem{prop}[thm]{Proposition}

\theoremstyle{remark}
\newtheorem{rem}[thm]{Remark}

\theoremstyle{definition}
\newtheorem{eg}[thm]{Example}
\newtheorem{defn}[thm]{Definition}

\newcommand{\grp}[1]{\langle #1 \rangle}				
\newcommand{\op}[2]{\mathrm{O}_{#1}(#2)}	

\newcommand{\ba}[1]{\overline{#1}}

\DeclareMathOperator{\Aut}{Aut}
\DeclareMathOperator{\Cay}{Cay}
\DeclareMathOperator{\Dih}{Dih}
\DeclareMathOperator{\Dic}{Dic}
\DeclareMathOperator{\Z}{Z}

\newcommand{\C}{\mathrm{C}}
\newcommand{\D}{\mathrm{D}}
\newcommand{\Q}{\mathrm{Q}}
\newcommand{\F}{\mathrm{F}}
\newcommand{\K}{\mathrm{K}}
\newcommand{\ZZ}{\mathbb{Z}}
\newcommand{\LL}{\mathcal{L}}

\newcommand{\soc}{\mathrm{soc}}
\newcommand{\AGL}{\mathrm{AGL}}
\newcommand{\PGL}{\mathrm{PGL}}

\newcommand{\PSL}{\mathrm{PSL}}

\renewcommand{\geq}{\geqslant}
\renewcommand{\leq}{\leqslant}

\renewcommand{\le}{\leqslant}

\title[CCA Sylow cyclic groups]{Characterising CCA Sylow cyclic groups whose order is not divisible by four}

\author{Luke Morgan}
\author{Joy Morris}
\author{Gabriel Verret}

\address{Luke Morgan and Gabriel Verret$^*$\\
Centre for the Mathematics of Symmetry and Computation, School of Mathematics and Statistics (M019)\\
The University of Western Australia\\
Crawley, 6009\\
Australia} 
\email{luke.morgan@uwa.edu.au}
\address{Joy Morris\\
Department of Mathematics and Computer Science\\
University of Lethbridge\\
Lethbridge, AB T1K 3M4\\
Canada}
\email[corresponding author]{joy.morris@uleth.ca}

\address{$*$ Current address: Department of Mathematics, The University of Auckland, Private Bag 92019, Auckland 1142, New Zealand.}
\email{g.verret@auckland.ac.nz}

\thanks{This research was supported by the Australian Research Council grants DE130101001 and DE160100081, by the Natural Science and Engineering Research Council of Canada, and by the Cheryl E. Praeger Fellowship from The University of Western Australia.}

\begin{document}

\begin{abstract}
A Cayley graph on a group $G$  has a natural edge-colouring. We say that such a graph is \emph{CCA} if every automorphism of the graph that preserves this edge-colouring is an element of the  normaliser of the  regular representation of $G$. A group $G$ is then said to be \emph{CCA} if every connected Cayley graph on $G$ is CCA. 

Our main result is a characterisation of non-CCA graphs on groups that are Sylow cyclic and whose order is not divisible by four. We also provide several new constructions of non-CCA graphs.

\end{abstract}

\maketitle

\section{Introduction}
All groups and all graphs in this paper are finite. Let $G$ be a group and let $S$ be an inverse-closed subset of $G$. The \emph{Cayley graph} of $G$ with respect to $S$ is the edge-coloured graph $\Cay(G,S)$ with vertex-set $G$ and, for every $g\in G$ and $s\in S$, an edge $\{g,sg\}$ with colour $\{s,s^{-1}\}$.  Its  group of colour-preserving automorphisms is    denoted  $\Aut_c(\Cay(G,S))$. Let $\Aut_{\pm 1}(G,S)=\{\alpha \in \Aut(G) \colon s^\alpha \in \{s,s^{-1}\}$ for all $s \in S \}$. It is easy to see that $G_R \rtimes\Aut_{\pm 1}(G,S)\leq \Aut_c(\Cay(G,S))$, where $G_R$ is the right-regular representation of $G$.

\begin{defn}[\cite{FirstCCA}]
The Cayley graph $\Cay(G,S)$ is \emph{CCA} (Cayley colour automorphism) if $\Aut_c(\Cay(G,S))=G_R \rtimes\Aut_{\pm 1}(G,S)$. The group $G$ is  \emph{CCA} if every connected Cayley graph on $G$ is CCA.
\end{defn}

In other words,  a Cayley graph  is CCA  if and only if the colour-preserving graph automorphisms are exactly the ``obvious'' ones. The terminology we use for this problem largely comes from \cite{FirstCCA}. Other papers that study this problem include \cite{CompleteCCA,CCASquarefree,VTCCA,CirculantCCA}.

Note that $\Cay(G,S)$ is connected if and only if $S$ generates $G$. It is also easy to see that $G_R \rtimes\Aut_{\pm1}(G,S)$ is precisely the normaliser of $G_R$ in $\Aut_c(\Cay(G,S))$. In particular, $\Cay(G,S)$ is CCA if and only if $G_R$ is normal in $\Aut_c(\Cay(G,S))$, c.f. \cite[Remark 6.2]{FirstCCA}.

In Section~\ref{prelim}, we introduce some basic terminology and recall a few previous results on the CCA property. 
In Section~\ref{sec:wr}, we consider wreath products of permutation groups, and produce  conditions that are sufficient to determine when such a product is a non-CCA group. This generalises results from~\cite{FirstCCA}. In Section~\ref{sec:constr}, we give some new constructions for non-CCA graphs.

Finally, in Section~\ref{sec:sylow cyclic}, we obtain a characterisation of non-CCA groups whose order is not divisible by four, in which every Sylow subgroup is cyclic. This generalises the work of \cite{CCASquarefree}, which dealt with the case of groups of odd squarefree order.

\section{Preliminaries}\label{prelim}

The identity of a group $G$ is denoted $1_G$, or simply $1$ if there is no risk of confusion.  We denote a dihedral group of order $2n$ by $\D_n$, while $\Q_8$ denotes the quaternion group of order $8$ with elements $\{\pm 1,\pm i, \pm j,\pm k\}$ and multiplication defined as usual.

We now state some preliminary results and introduce some terminology related to Cayley graphs. Let $\Gamma$ be a graph and let $v$ be a vertex of $\Gamma$. The neighbourhood of $v$ is denoted by $\Gamma(v)$. If $A$ is a group of automorphisms of $\Gamma$, then the permutation group induced by the vertex-stabiliser $A_v$ on the neighbourhood of $v$ is denoted $A_v^{\Gamma(v)}$.

\begin{lem}[{\cite[Lemma~6.3]{FirstCCA}}]\label{VxStab2Group}
The vertex-stabiliser in the colour-preserving group of automorphisms of a connected Cayley graph is a $2$-group.
\end{lem}

\begin{defn}
Let $G$ be a group, let $\Gamma=\Cay(G,S)$ and let $N$ be a normal subgroup of $G$. The \emph{quotient graph} $\Gamma/N$ is $\Cay(G/N,S/N)$, where $S/N=\{sN\colon s\in S\}$.
\end{defn}

\begin{lem}[{\cite[Lemma~3.4]{CCASquarefree}}]\label{QuotientGraph}
Let $A$ be a colour-preserving group of automorphisms of $\Cay(G,S)$, let $N$ be a normal subgroup of $A$ and let $K$ be the kernel of the action of $A$ on the $N$-orbits. If $N\leq G$, then $A/K$ is a colour-preserving group of automorphisms of $\Gamma/N$.
\end{lem}

\begin{lem}\label{coolnew}
Let $\Gamma=\Cay(G,S)$, let $A$ be a colour-preserving group of automorphisms of $\Gamma$, let $N$ be a normal $2$-subgroup of $A$ and let $K$ be the kernel of the action of $A$ on the $N$-orbits. If $K_v\neq 1$, then $S$ contains an element whose order is a power of $2$ that is at least $4$.
\end{lem}
\begin{proof}
Let $v_0=v$. Since $K_{v_0}\neq 1$, $K_{v_0}^{\Gamma(v_0)}\neq 1$ and there exists $k\in K_{v_0}$ and a neighbour $u_0$ of $v_0$ such that $u_0^k \neq u_0$. Let $u_1=u_0^k$. Note that $K_{u_1}\neq K_{v_0}$ hence there exists $\ell \in K_{u_1}$ such that  $v_0^\ell \neq v_0$. Let $v_1=v_0^\ell$. Repeating this process, we get a monochromatic cycle $C=(u_0,v_0,u_1,v_1,...)$  of length at least $3$. By construction, $u_i\in u_0^K=u_0^N$ and $v_i\in v_0^K=v_0^N$ for all $i$. In particular,  $|C\cap v_0^N|\in \{|C|,|C|/2\}$. Since each vertex of $\Gamma$ lies in a unique monochromatic cycle  of a given colour, $C$ is a block for $A$. On the other hand, $v_0^N$ is also a block for $A$ and thus so is $C\cap v_0^N$. It follows that $|C\cap v_0^N|$ divides $|N|$ which is a power of $2$. This implies that $|C|$ is also a power of $2$. Since $|C|\geq 3$, $|C|$ is divisible by $4$ and the result follows from the fact that $C$ is monochromatic.
\end{proof}

For a group $H$, let $H^2 :=\langle x^2 \mid x\in H \rangle$. 
The following lemma is inspired by an argument contained within \cite[Theorem 6.8]{FirstCCA}.

\begin{lem}\label{newnewyep}
Let $\Gamma=\Cay(G,S)$ be connected, let $A$ be a colour-preserving group of automorphisms of $\Gamma$ that is normalised by $G$ and let $v$ be a vertex of $\Gamma$. If $A_v$ has a subgroup $U$ such that $U\leq (A_v)^2$ and no other subgroup of $A_v$ is isomorphic to $U$, then $U=1$. In particular, $A_v$ is isomorphic to neither $\ZZ_{2^n}$ for $n \geqslant 2$ nor isomorphic to $\D_{2^n}$ for  $n\geqslant 3$.
\end{lem}
\begin{proof}
Since $A$ is colour-preserving, $A_v^{\Gamma(v)}$
  is an elementary abelian $2$-group.
 Since $U\leq (A_v)^2$, it follows that $U$ fixes all the neighbours of $v$.
Let $s\in S$. Since $A$ is normalised by $G$, we have $U^s\leq A$ and, by the previous  observation, $U^s\leq A_v$. As $A_v$ has a unique subgroup isomorphic to $U$, we must have $U=U^s$. Since this holds for every $s\in S$, $U$ is normalised by $G$. As $G$ is transitive and $U$ fixes $v$, this implies that $U=1$.

The second part of the lemma follows from the first. 
Indeed, if $A_v$ is isomorphic to $\ZZ_{2^n}$ for $n\geqslant 2$ or to $\D_{2^n}$  for $n\geqslant 3$, then $(A_v)^2$ is non-trivial and is the unique cyclic subgroup of its order.
\end{proof}

\section{Wreath products}\label{sec:wr}

\begin{prop}\label{wreath}
Let $H$ be a permutation group on a set $\Omega$, let $G$ be a group and let $X=G \wr_{\Omega} H$. If 
\begin{itemize}
\item there is an inverse-closed generating set $S$ for $G$ and a non-identity bijection $\tau: G \rightarrow G$ such that $\tau$ fixes  $1$, and $\tau(sg)=s^{\pm 1}\tau(g)$ for every $g \in G$ and  every $s \in S$, and
\item either $H$ is nontrivial or $\tau\not\in \Aut(G)$, 
\end{itemize}then $X$ is non-CCA.
\end{prop}
\begin{proof}
Let $m=|\Omega|$ and write $\Omega=\{1,\ldots, m\}$ such that, if $H$ is nontrivial, then $1$ is not fixed by $H$. 
 
Write $X=H \ltimes (G_{1}\times\cdots\times G_{m})$. Note that, if $g\in G_{i}$ and $h\in H$, then $g^h \in G_{{i}^h}$. Without loss of generality, we may assume that $1_G\notin S$. Let $S_i$ be the subset of $G_{i}$ corresponding to $S$, let $T=(H-\{1_H\}) \cup S_1 \cup \cdots \cup S_m$ and let $\Gamma=\Cay(X,T)$. Note that $T$ generates $X$ hence $\Gamma$ is connected. We will show that $\Gamma$ is non-CCA.

Define $\tau':X \to X$ by $\tau' : hg_1g_2\cdots g_m \mapsto h \tau(g_1)g_2 \cdots g_m$, where $g_i\in G_{i}$ and $h\in H$. Let $v$ be a vertex of $\Gamma$ and let $s\in T$. We will show that $\tau'(sv)=s^{\pm 1} \tau'(v)$ and hence $\tau'$ is a colour-preserving automorphism of $\Gamma$.  Write $v = hg_1\cdots g_m$ with  $h \in H$ and $g_i\in G_{i}$. Let $g=g_1\cdots g_m$. Note that $\tau'(v) =\tau'(hg)= h\tau'(g)$. If $s\in H$, then $\tau'(sv)=\tau'(shg) = sh\tau'(g)=s\tau'(v)$. Suppose now that $s\in S_i$ for some $i\in\Omega$. If $i^h\neq 1$, then 
$$\tau'(s^hg) = \tau'(g_1\cdots s^h g_{i^h} \cdots g_m)=\tau(g_1) g_2\cdots s^h g_{i^h} \cdots g_m=s^h\tau(g_1) g_2 \cdots g_m=s^h\tau'(g).$$ 
If $i^h=1$, then  $s^h\in S_1\subseteq T$  and
$$\tau'(s^hg) = \tau'(s^hg_1 \cdots g_m)=\tau(s^hg_1) g_2\cdots  g_m=(s^h)^{\pm 1}\tau(g_1) g_2 \cdots g_m=(s^h)^{\pm 1}\tau'(g).$$ 
Either way, we have
$$\tau'(sv) = \tau'(hs^hg) = h\tau'(s^hg) = h (s^h)^{\pm1}\tau'(g) = s^{\pm1} h \tau'(g) = s^{\pm1} \tau'(hg)=s^{\pm1}\tau'(v).$$ This completes the proof that  $\tau'$ is a colour-preserving automorphism of $\Gamma$. It remains to show that $\tau'$ is not a group automorphism of $X$.  (Note that $\tau'$ fixes $1_X$, so if $\tau' \in X_R \rtimes\Aut_{\pm1}(X,T)$, then $\tau' \in \Aut(X)$.)

If $H$ is nontrivial, then, since $1$ is not fixed by $H$, there exists $h\in H$ such that $1^{h}\neq 1$. Let $g$ be an element of $G_1$  that is not fixed by $\tau$. We have $\tau'(gh) = \tau'(hg^h)  = h g^h=gh$ but $\tau'(g)\tau'(h) = \tau(g) h$. Since $g\neq \tau(g)$, $\tau'$ is not an automorphism of $X$.

If $H$ is trivial and $\tau\not\in \Aut(G)$, then there exist $g_1, g_2 \in G$ such that $\tau(g_1g_2)\neq \tau(g_1)\tau(g_2)$. Applying $\tau'$ to the corresponding elements of $G_1$ shows that $\tau'$ is not an  automorphism of $X$. This completes the proof.
\end{proof}

We now obtain a few corollaries of Proposition~\ref{wreath}.

\begin{cor}
Let $H$ be a permutation group on a set $\Omega$ and let $G$ be a group.
If $G$ is non-CCA, then $G\wr_\Omega H$ is non-CCA.
\end{cor}

\begin{proof}
Since $G$ is non-CCA, there exists a colour-preserving graph automorphism $\tau$ of a Cayley graph $\Cay(G,S)$ such that $\tau(1_G)=1_G$ but $\tau$ does not normalise $G_R$. Since $\tau$ is colour-preserving, $\tau(sg)=s^{\pm1}\tau(g)$ for every $g \in G$ and every $s \in S$. Finally, since $\tau$ does not normalise $G_R$, we have $\tau\not\in \Aut(G)$ and the result follows from  Proposition~\ref{wreath}.
\end{proof}

\begin{cor}
Let $H$ be a nontrivial permutation group on a set $\Omega$ and let $G$ be a group. If $G=B\ltimes A$, where $A$ is  abelian of exponent greater than $2$, then $G \wr_\Omega H$ is non-CCA.
\end{cor}

\begin{proof}
Every element of $G$ can be written uniquely as $ba$ with $a\in A$ and $b\in B$. Let $\tau$ be the permutation of $G$ mapping $ba$ to $ba^{-1}$. Clearly, $\tau$ fixes $1_G$ but, since $A$ has exponent greater than $2$, $\tau$ is not the identity. Let $S=(A \cup B) -\{1_G\}$. Note that $S$ is an inverse-closed generating set for $G$. Let $g\in G$, let $s\in S$ and write $g=ba$ with $a\in A$ and $b\in B$. If $s\in B$, then $\tau(sg)=\tau(sba)=sba^{-1}=s\tau(ba)=s\tau(g)$. Otherwise, $s\in A$, $s^b\in A$ and
$$\tau(sg)=\tau(sba)=\tau(bs^ba) = b(s^ba)^{-1} = b(s^b)^{-1}a^{-1} =s^{-1}ba^{-1}=s^{-1}\tau(g).$$
The result then follows from  Proposition~\ref{wreath}, since $H$ is nontrivial.
\end{proof}

\begin{cor}\label{cor:cor}
Let $H$ be a permutation group on a set $\Omega$ and let $G$ be a group. If 
\begin{itemize}
\item $G$ has exponent greater than $2$, 
\item $H$ is nontrivial when $G$ is abelian, and
\item  $G$ has a generating set $S$ with the property that $s^g=s^{\pm1}$ for every $s\in S$ and $g\in G$, 

\end{itemize}then $G \wr_\Omega H$ is non-CCA.
\end{cor}

\begin{proof}
We can assume without loss of generality that $S$ is inverse-closed. Let $\tau$ be the permutation of $G$ that maps every element to its inverse. For every $s\in S$ and $g\in G$, we have $s^g=s^{\pm 1}$ and thus $\tau(sg)=g^{-1}s^{-1}=s^{\pm 1}g^{-1}=s^{\pm1}\tau(g)$. Since $G$ has exponent greater than $2$, $\tau$ is not the identity. If $H$ is trivial, then $G$ is non-abelian so that $\tau$ is not an automorphism of $G$. The result then follows from  Proposition~\ref{wreath}.
\end{proof}

In view of Corollary~\ref{cor:cor}, it would be interesting to determine the groups $G$ such that  $G$ has a generating set $S$ with the property that $s^g=s^{\pm1}$ for every $s\in S$ and $g\in G$. This family of groups includes abelian groups and $\Q_8$.  This family is closed under central products but it also includes examples which do not arise as central products of smaller groups in the family, for example  the extraspecial group of order $32$ and minus type.

\section{A few constructions for non-CCA graphs}\label{sec:constr}

In this section, we will describe a few constructions which yield non-CCA Cayley graphs. For a group $G$, let $\K_G$   denote $\Cay(G,G-\{1\})$, the complete Cayley graph on $G$. 
We will need a result which tells us when $  G_R < \Aut_c(\K_G)$. First we state some definitions.

\begin{defn}\label{def:dihedral}
Let $A$ be an abelian group of exponent greater than $2$, and define a map $\iota:A \to A$ by $\iota(a)=a^{-1}$ for every $a \in A$.
The \emph{generalised dihedral group} over $A$ is $\Dih(A)=A\rtimes\langle\iota\rangle$. 
\end{defn}

\begin{defn} \label{iota}
Let $A$ be an abelian group of even order and of exponent greater than $2$, and let $y$ be an element of $A$ of order $2$. The \emph{generalised dicyclic group} over $A$ is  $\Dic(A,y):=\langle A,x\mid x^2=y, a^x=a^{-1}\ \forall a \in A\rangle$. 
Let $\iota$ be the permutation of $\Dic(A,y)$ that fixes $A$ pointwise and maps every element of the coset $Ax$ to its inverse.
\end{defn}

It is not hard to check that $\iota$ is an automorphism of $\Dic(A,y)$. 

\begin{defn}\label{sigma_i}
For $\alpha\in\{i,j,k\}$, let  $S_\alpha=\{\pm \alpha\}\times \ZZ_2^n$ and let $\sigma_\alpha$ be the permutation of $\Q_8 \times \ZZ_2^n$ that inverts every element of $S_\alpha$ and fixes every other element.
\end{defn}

\begin{thm}[\cite{CompleteCCA}, Classification Theorem]\label{CompleteGraph}
If $G$ is a group, then $G_R<\Aut_c(\K_G)$ if and only if one of the following occurs.
\begin{enumerate}
\item $G$ is abelian but is not an elementary abelian $2$-group, and $\Aut_c(\K_G)=\Dih(G)$, \label{caseone}
\item $G$ is generalised dicyclic  but not of the form $\Q_8\times\ZZ_2^n$, and $\Aut_c(\K_G)=G_R\rtimes\langle\iota\rangle$, where $\iota$ is as in Definition~\ref{iota}, or \label{casetwo}
\item $G\cong\Q_8\times\ZZ_2^n$ and $\Aut_c(\K_G)= \grp{G_R, \sigma_i,\sigma_j,\sigma_k}$, where $\sigma_i, \sigma_j, \sigma_k$ are as in Definition~\ref{sigma_i}.\label{Complete-case3}
\end{enumerate}
\end{thm}

\begin{defn}
Let $B$ be a permutation group and let $G$ be a regular subgroup of $B$. We say that $(G,B)$ is a \emph{complete colour pair} if $G$ is as  in the conclusion of Theorem~\ref{CompleteGraph} and $B \le \Aut_c(\K_G)$.
\end{defn}

For a graph $\Gamma$, let $\LL(\Gamma)$ denote its line graph.

\begin{prop}\label{LineGraph}
Let $\Gamma$ be a  connected bipartite $G$-edge-regular graph. If $H$ is a group of automorphisms of $\Gamma$ such that:
\begin{itemize}
\item $G\leq H$,
\item the orbits of $H$ on the vertex-set of $\Gamma$ are exactly the biparts, and
\item  for every vertex $v$ of $\Gamma$, either 
\begin{itemize}
\item $G_v^{\Gamma(v)}=H_v^{\Gamma(v)},$ or
\item $(G_v^{\Gamma(v)},H_v^{\Gamma(v)})$ is a complete colour pair,
\end{itemize} 
\end{itemize}
then $H$ is a colour-preserving group of automorphisms of $\LL(\Gamma)$ viewed as a Cayley graph on $G$.
\end{prop}
\begin{proof}
Since $G$ acts regularly on edges of $\Gamma$, its induced action on $\LL(\Gamma)$ is regular on vertices. Vertices of $\Gamma$ induce cliques in $\LL(\Gamma)$, which we call \emph{special}. Clearly, $H$ has exactly two orbits on special  cliques. Moreover, special cliques partition the edges of $\LL(\Gamma)$, and each vertex of $\LL(\Gamma)$ is in exactly two special cliques, one from each $H$-orbit. Since $G\leq H$, the set of edge-colours appearing in special cliques from different $H$-orbits is disjoint.

Let $v$ be a vertex of $\Gamma$ and let $C$ be the corresponding special clique of $\LL(\Gamma)$. Note that $H_v^{\Gamma(v)}$ is permutation isomorphic to $H_C^C$, while $G_v^{\Gamma(v)}\cong G_C^C\cong G_C$. Since $G$ is vertex-regular on $\LL(\Gamma)$, $G_C$ is regular on $C$ and thus $C$ can be viewed as a complete Cayley graph on $G_C$.  If $(G_v^{\Gamma(v)},H_v^{\Gamma(v)})$ is a complete colour pair, Theorem~\ref{CompleteGraph} implies that $H_C^C$ is colour-preserving.   If $G_v^{\Gamma(v)}=H_v^{\Gamma(v)},$ then since $G$ is colour-preserving, so is $H_C^C$.  Since $G$ acts transitively on the special cliques within an $H$-orbit and $G$ is colour-preserving, it follows that $H$ is colour-preserving.
\end{proof}

\begin{rem}
In the proof of Proposition~\ref{LineGraph}, we only use one direction of Theorem~\ref{CompleteGraph}, namely that if $G$ appears in Theorem~\ref{CompleteGraph}, then  $G_R<\Aut_c(\K_G)$. The converse is not used here, but it  can help to identify situations where Proposition~\ref{LineGraph} can be used to construct non-CCA graphs.
\end{rem}

\begin{eg}\label{eg:f21}
Let $\Gamma$ be the Heawood graph and let $H$ be the bipart-preserving subgroup of $\Aut(\Gamma)$. Note that $H\cong\PSL(2,7)$ and $H$ contains an edge-regular subgroup $G$ isomorphic to $\F_{21}$, the Frobenius group of order $21$. Moreover, for every vertex $v$ of $\Gamma$, we have $G_v^{\Gamma(v)}\cong\ZZ_3$ while $H_v^{\Gamma(v)}\cong\D_3$ and $(\ZZ_3,\D_3)$ is a complete colour pair. By Proposition~\ref{LineGraph}, $H$ is a colour-preserving group of automorphisms of $\LL(\Gamma)$ viewed as a Cayley graph on $G$. Since $G$ is not normal in $H$, it follows that $\LL(\Gamma)$ is a non-CCA graph and so $\F_{21}$ is a non-CCA group. 
\end{eg}

Example~\ref{eg:f21} was previously studied  in~\cite{CCASquarefree} and \cite{FirstCCA}, under a slightly different guise.

\begin{eg}
Let $A \cong \Q_8 \times \ZZ_2^m$ and $B = \Aut_c(K_A) =\langle A_R, \sigma_i, \sigma_j, \sigma_k\rangle$. Then $A$ is not normal in $B$, and by Theorem~\ref{CompleteGraph}(\ref{Complete-case3}), $(A,B)$ is a complete colour pair. Let $n=|A|$ and let $\K_{n,n}$ be the complete bipartite graph of order $2n$. Let $G= A\times A$ and let $H=B\times B$. By Proposition~\ref{LineGraph}, $H$ is a colour-preserving group of automorphisms of $\LL(\K_{n,n})$ viewed as a Cayley graph on $G$. Since $A$ is not normal in $B$, $G$ is not normal in $H$ hence $\LL(\K_{n,n})$ is a non-CCA graph and so $G$ is a non-CCA group.
\end{eg}

For a graph $\Gamma$, let $\mathcal S(\Gamma)$ denote its subdivision graph.

\begin{cor}\label{Subdivision}
Let $\Gamma$ be a  connected $G$-arc-regular graph. If $H$ is a group of automorphisms of $\Gamma$ such that:
\begin{itemize}
\item $G\leq H$, and
\item $(G_v^{\Gamma(v)},H_v^{\Gamma(v)})$ is a complete colour pair for every vertex $v$ of $\Gamma$,
\end{itemize}
then $H$ is a colour-preserving group of automorphisms of $\LL(\mathcal S(\Gamma))$ viewed as a Cayley graph on $G$.
\end{cor}
\begin{proof}
Let $\Gamma'=\mathcal S(\Gamma)$. We show that Proposition~\ref{LineGraph} applies to $\Gamma'$. Clearly, $\Gamma'$ is bipartite and $G$ acts on it faithfully and edge-regularly. It is also obvious that, in its induced action on $\Gamma'$, $H$ must preserve the biparts of $\Gamma'$. Finally, let $x$ be a vertex of $\Gamma'$. If $x$ arose from a vertex $v$ of $\Gamma$, then we have that $A_v^{\Gamma(v)}$ is permutation isomorphic to $A_x^{\Gamma'(x)}$ for every $A\leq\Aut(\Gamma)$. Since $(G_v^{\Gamma(v)},H_v^{\Gamma(v)})$ is a complete colour pair, so is $(G_x^{\Gamma'(x)},H_x^{\Gamma'(x)})$. If $x$ arose from an edge of $\Gamma$, then $x$ has valency $2$ and, since $G$ is arc-transitive, $G_x^{\Gamma'(x)}=H_x^{\Gamma'(x)}\cong\ZZ_2$ and $(G_x^{\Gamma'(x)},H_x^{\Gamma'(x)})$ is a complete colour pair.
\end{proof}

\begin{eg}\label{eg:42}
Let $\Gamma$ be the Heawood graph and let $H=\Aut(\Gamma)$. Note that $H$ contains an arc-regular subgroup $G$ isomorphic to $\AGL(1,7)$. Moreover,  for every vertex $v$ of $\Gamma$, we have $G_v^{\Gamma(v)}\cong\ZZ_3$ while $H_v^{\Gamma(v)}\cong\D_3$. By Corollary~\ref{Subdivision}, $H$ is a colour-preserving group of automorphisms of $\LL(\mathcal S(\Gamma))$ viewed as a Cayley graph on $G$. Since $G$ is not normal in $H$, it follows that $\Gamma$ is a non-CCA  graph and so $\AGL(1,7)$ is a non-CCA group. 
\end{eg}

\begin{rem} In fact, $\AGL(1,7)$ is a Sylow cyclic group whose order is not divisible by  four, so Example~\ref{eg:42} will appear again in our characterisation of non-CCA groups of this sort, in Section~\ref{sec:sylow cyclic}. However, the construction we have just presented is very different from the approach we use in that section.
\end{rem}

\section{Sylow cyclic and order not divisible by four}\label{sec:sylow cyclic}
  
We first introduce some notation that will be useful throughout this section. Recall that $\PGL(2,7)$ has a unique conjugacy class of subgroups isomorphic to $\AGL(1,7)$. The intersection of such a subgroup with the socle $\PSL(2,7)$ is a Frobenius group of order $21$ which we will denote $\F_{21}$. We say that a group $G$ is \emph{Sylow cyclic} if, for every prime $p$, the Sylow $p$-subgroups of $G$ are cyclic. 

Our aim in this section is to characterise both the non-CCA Sylow cyclic groups whose order is not divisible by four, and the structure of the corresponding colour-preserving automorphism groups for non-CCA graphs.

\begin{thm}\label{thm:square-free}
Let $G$ be a Sylow cyclic group  whose order is  not divisible by four, let $\Gamma=\Cay(G,S)$, let $A$ be a colour-preserving group of automorphisms of $\Gamma$, let $R$ be a Sylow $2$-subgroup of $G$ and let $r$ be a generator of $R$. If $G$ is not normal in $A$,  then
 $G=(F\times H)\rtimes R$ and $A = (T\times J)\rtimes R$, where the following hold:

\begin{itemize}
\item[(i)] $\PSL(2,7)\cong T\trianglelefteq A$,
\item[(ii)] $T\cap G=F\cong\F_{21}$,
\item[(iii)] $J\cap G=H\trianglelefteq J\trianglelefteq A$,
\item[(iv)] $H$ is self-centralising in $J$,
\item[(v)] $J$ splits over $H$,
\item[(vi)] $H$ is normal in $A$.
\end{itemize}

\end{thm}
\begin{proof} 
To avoid ambiguity, for $g\in G$, we write $[g]$ for the vertex of $\Gamma$ corresponding to $g$ and, for $X\subseteq G$, we write $[X]$ for $\{[x]\colon x\in X\}$. 

 Let $P$ be a Sylow $2$-subgroup of $A$ containing $R$. By Lemma~\ref{VxStab2Group}, $A_{[1]}$ is a $2$-group. Up to relabelling, we may assume that $A_{[1]}\leq P$. Since $G$ is regular, we have $A = GA_{[1]}$ and $|A|=|G||A_{[1]}|$.  Note that (v) and (vi) follow from the rest of the claims. Indeed, $H$ must have odd order and, since $|A:G|$ is a power of $2$, so is $|J:H|$ and thus $J =H \rtimes (P \cap J)$. As $H$ has odd order and is normal in $J$, it must be characteristic in $J$ and thus normal in $A$.

Since $|G|$ is not divisible by $4$, it follows that $G$ has a characteristic subgroup 
$G_2$
of odd order 
such that $G=G_2\rtimes R$. By order considerations, we have $A=G_2P$.

\textbf{Case 1:} There is no minimal normal subgroup of $A$ of odd order.

In this case, we have that $\soc(A)=T_1\times\cdots \times T_k\times B$, where $\soc(A)$ is the socle of $A$,  the $T_i$s are non-abelian simple groups, and $B$ is an elementary abelian $2$-group.  Recall that $A=G_2P$, that is, $A$ has a $2$-complement.
Since this property is inherited by normal subgroups, $\soc(A)$ and $T_i$ also have $2$-complements for every $i$. This implies that, for every $i$, $T_i \cong \PSL(2,p)$ for some Mersenne prime $p$ (see~\cite[Theorem~1.3]{2Complement} for example). Now, $|T_i|$ is divisible by $3$ but the Sylow $3$-subgroup of $\soc(A)$ is cyclic (since $|A:G|$ is a power of $2$ and $G$ is Sylow cyclic)  so that $k=1$. Let $T=T_1$. Suppose that $p>7$ and hence $p\geq 31$. Note that $T_{[1]}$ has index at most $2$ in some Sylow $2$-subgroup of $T$ which is isomorphic to $\D_{(p+1)/2}$. It follows that $T_{[1]}$ has order at least $(p+1)/2$ and is either dihedral or cyclic. Since $p\geq 31$, this implies that $T_{[1]}$ contains a unique cyclic subgroup of order $(p+1)/8$, say $U$, and $U$ is contained in $(T_{[1]})^2$.
 By Lemma~\ref{newnewyep}, $U=1$, which is a contradiction. 
It follows that $p=7$ and $T\cong\PSL(2,7)$. 

Let $\op{2}{A}$ be the largest normal $2$-subgroup of $A$.  
If $\op{2}{A}=1$, then $\soc(A)=T\cong\PSL(2,7)$ and $A$ is isomorphic to one of $\PSL(2,7)$ or $\PGL(2,7)$.  If $A \cong \PSL(2,7)$, then, as $\F_{21}$ is the only proper subgroup of $\PSL(2,7)$ with index a power of $2$, $G\cong\F_{21}$ and the theorem holds. If $A \cong \PGL(2,7)$, then, for the same reason, $G$ is isomorphic to either $\F_{21}$ or $\AGL(1,7)$.
If $G\cong \F_{21}$, then $A_{[1]}$ must be a Sylow $2$-subgroup of $A$ and thus isomorphic to $\D_8$. In particular, $A_{[1]}$ contains a unique cyclic subgroup of order $4$ and this subgroup is contained in $(A_{[1]})^2$. This contradicts Lemma~\ref{newnewyep}. We must therefore  have $G\cong\AGL(1,7)$ and again the theorem holds.

 We now assume that 
$\op{2}{A} \neq 1$.
In particular, the orbits of $\op{2}{A}$ are of equal length, which is a power of $2$ greater than $1$. It follows that $|\op{2}{A}:\op{2}{A}_{[1]}|=|[1]^{\op{2}{A}}|= 2$. 
 Let $K$ be the kernel of the action of $ A$ on the $\op{2}{A}$-orbits. By Lemma \ref{coolnew}, $K$ is semiregular hence so is  $\op{2}{A}$. It follows that  $|\op{2}{A}|=2$ and $\op{2}{A}$ is central in $A$. This implies that $B=\op{2}{A}$, hence $\soc(A)=T\times \op{2}{A}$.
Now, $A_{[1]}$ is a complement for $\op{2}{A}$ in $P$, so by Gaschutz' Theorem (see for example~\cite[3.3.2]{Stellmacher}), $\op{2}{A}$ has a complement in $A$.

Clearly,  $\op{2}{A}\leq\C_A(T)$. We show that equality holds. Suppose, on the contrary, that $\op{2}{A}<\C_A(T)$. Since $\C_A(T)$ is normal in $A$, $\C_A(T)/\op{2}{A}$ must contain a minimal normal subgroup of $A/\op{2}{A}$, say $Y/\op{2}{A}$. Since $\op{2}{A}$ has a complement in $A$, $\op{2}{A}$ has a complement in $Y$, say $Z$. Thus $Y=\op{2}{A}\times Z$ and $Z$ is isomorphic to $Y/\op{2}{A}$ which is a minimal normal subgroup of $A/\op{2}{A}$ and therefore either an elementary abelian group of odd order, or a product of non-abelian simple groups. It follows that $Z$ is characteristic in $Y$ and thus normal in $A$. Since the action of $A$ by conjugation on $Z$ and on $Y/\op{2}{A}$ are equivalent, we see that $Z$ is a minimal normal subgroup of $A$. The only possibility is that $Z=T$ but, since $T$ has trivial centre, this contradicts the fact that $Z\leq Y\leq \C_A(T)$. This concludes our proof that $\C_A(T)=\op{2}{A}\cong\ZZ_2$. 

As $\op{2}{A}$ has a complement in $A$, it follows that $A$ is isomorphic to one of $\PSL(2,7)\times \ZZ_2$ or $\PGL(2,7)\times\ZZ_2$. Suppose first that $A \cong   \PSL(2,7)\times\ZZ_2$. Since $G$ has even order, is not normal in $A$ and has index a power of $2$, we must have $G\cong \F_{21}\times \ZZ_2$ and the theorem holds with $H=J=1$. Finally, suppose that  $A \cong \PGL(2,7)\times \ZZ_2$. In particular, $P=Q\times\op{2}{A}$ where $Q\cong\D_8$. Note that $|P:A_{[1]}|=2$ and $A_{[1]} \cap \op{2}{A} =1$ hence $A_{[1]} \cong P/\op{2}{A}\cong \D_8$. This contradicts Lemma~\ref{newnewyep}.

\textbf{Case 2:} There exists a minimal normal subgroup of $A$ of odd order.

Let $N$ be a minimal normal subgroup of odd order, that is, $|N|$ is a power of some odd prime $p$. Let $K$ be the kernel of the action of $A$ on the set of $N$-orbits. Since the $N$-orbits have odd size and $K_{[1]} \le A_{[1]} $ is a $2$-group, $K_{[1]}$ must fix at least one point in every $N$-orbit. 
For each $N$-orbit $B$, pick $b\in G$ such that  $K_{[1]}=K_{[b]}$ and $B=[b]^N$. Now the kernel of the action of $K$ on $[1]^N$ is 
$K_{([1]^N)}=\bigcap_{n\in N} (K_{[1]})^n = \bigcap_{n\in N} (K_{[b]})^n = K_{([b]^N)}.$
It follows that $K_{([1]^N)}$ fixes every vertex of $\Gamma$, and so $K$ acts faithfully on $[1]^N$.
Moreover, $N_{[1]}=1$ hence $K=N\rtimes K_{[1]}$.  As $|A:G|$ is a power of $2$, $G$ contains a Sylow $p$-subgroup of $A$. Since  $N$ is normal in $A$, it is contained in every Sylow subgroup of $A$, and thus $N\leq G$. We therefore have $GK=GNK_{[1]}=  GK_{[1]}$. Since $G$ is Sylow cyclic and $N$ is elementary abelian, we must have $|N|=p$.

If $K_{[1]}\neq 1$, then $K_{[1]}$ must move a neighbour of $[1]$, say $[s]$ for some non-involution $s\in S$. It follows that $K_{[s]}$ must move $[1]$, necessarily to $[s^2]$ since $K$ is colour-preserving, and thus $[s^2]\in [1]^N$. Let $C$ be the cycle containing $[1]$ with edge-label $\{s,s^{-1}\}$. We have shown that $[1],[s^2]\in [1]^N\cap C$ and hence $|[1]^N\cap C|\geq 2$. Since $[1]^N$ and $C$ are both blocks for the action of $G$, the former of prime order, it follows that $[1]^N\cap C=[1]^N$, that is $[1]^N\subseteq C$. Since $K$ acts faithfully on $[1]^N$, $K_{[1]}$ acts faithfully on $C$ and thus $|GK:G|=|K_{[1]}|\leq 2$. It follows that $G$ is normal in $GK$, $GK=G\rtimes K_{[1]}$ and either $K=N\cong\ZZ_p$ or $K\cong\D_p$.

Suppose that $GK/K$ is normal in $A / K$ and hence $GK$ is normal in $A$. We show that this implies that $G$ is normal in $A$, which is a contradiction. This is trivial if $G=GK$ hence we assume that $|GK:G|=2$ and $K\cong\D_p$. If $G$ has odd order, then it is characteristic in $GK$ and thus normal in $A$. We may thus assume that $G$ has even order. 
Recall that $G_2$ is a characteristic subgroup of index $2$ in $G$, hence $G_2$ is normal in $GK$ and, since $|GK:G_2|=4$, we have that  $G_2$ is characteristic in $GK$ and thus normal in $A$. Note that $G$ and $G_2\rtimes K_{[1]}$ both have index two in $GK$ but $G_2\rtimes K_{[1]}$ is not semiregular, hence they are not conjugate in $A$. In particular, $G_2\rtimes K_{[1]}$ and $G$ are distinct index two subgroups of $GK$ and thus $GK/G_2$ is elementary abelian of order $4$. Let $X$ be the centraliser of $N$ in $GK$. Since $N\cong\ZZ_p$, $\Aut(N)$ is cyclic hence $GK/X$ is cyclic and $X$ is not contained in $G_2$. Since $N$, $G_2$ and $GK$ are normal in $A$, so is $XG_2$. If $XG_2 = G$, then we are done. We thus assume that this is not the case. Note that $|XG_2:X|=|G_2 :\C_{G_2}(N)|$ is odd,  hence every Sylow $2$-subgroup of $XG_2$ centralises $N$. Since $K_{[1]}$ has order $2$ but does not centralise $N$, $G_2\rtimes K_{[1]}$ is not contained in $XG_2$.   We thus conclude that $G$, $G_2K_{[1]} $ and $XG_2$ are the three index two subgroups of $GK$ containing $G_2$. One of them is normal in $A$, and we have seen that  the other two are not conjugate in $A$. It follows that all three are normal in $A$. In particular, $G$ is normal in $A$, a contradiction.

We may thus assume that  $GK/K$ is not normal in $A / K$. Again, we use the bar notation with respect to the mapping $A\mapsto A/K$. By Lemma~\ref{QuotientGraph}, $\ba{A}$ is a colour-preserving group of automorphisms of $\Gamma/N$. By induction, we have $\ba{G}=(\ba{F}\times \ba{H})\rtimes \ba{D}$ and $\ba{A} = (\ba{T}\times \ba{J})\rtimes \ba{D}$, where $\overline{D}$ is a Sylow $2$-subgroup of $\overline{G}$,  $\PSL(2,7)\cong \ba{T} \trianglelefteq \ba{A}$, $\ba{T} \cap \ba{G} = \ba{F} \cong \F_{21}$, $\ba{J} \cap \ba{G} =\ba{H} \trianglelefteq \ba{J} \trianglelefteq \ba{A}$ and 
$\ba{H}$ is self-centralising in $\ba{J}$. Further, since $R \cap K = 1$ we may assume $\ba{R}=\ba{D}$.   Note that $T / \C_T(K)  \le \Aut(K)$ is soluble since $K$ is either cyclic or dihedral.  As $T/K\cong\PSL(2,7)$, it follows that $T=K\C_T(K)$ and hence $\C_T(K)/\Z(K)\cong\PSL(2,7)$. If $K \cong \D_p$, then $\Z(K) = 1$. Set $T_0=\C_T(K)$ in this case. Otherwise, $\Z(K)=N=K\cong\ZZ_p$ and, since the Schur multiplier of $\PSL(2,7)$ has order $2$, we have $\C_T(K)=N\times T_0$ for some $T_0$. In both cases, $T_0\cong\PSL(2,7)$ and, since both $T$ and $K$ are normal in $A$, so is $T_0$, which proves (i). Now, $TJ = T_0KJ = T_0 J$, both $T_0$ and $J$ are normal in $A$ and $T_0\cap J=1$ hence $A=(T_0 \times J)\rtimes R$. Since $\ba{T_0} = \ba{T}$ there is  $F_0 \leqslant T_0$ such that $\ba{F_0} = \ba{F}$. Since $F_0\cap K=1$ we have $F_0 \cong\ba{F_0}\cong \F_{21}$. Further, $F=F_0K =F_0 \times K$. Since $|GK:G|\leq 2 $, we have that $F_0 \leqslant G$ and, since $F_0$ is maximal in $T_0$, we have $T_0 \cap G = F_0$ which is (ii).

Note that $|H|$ is not divisible by $4$, hence $H=H_0\rtimes K_{[1]}$ for some characteristic subgroup $H_0$ of $H$. In particular, $\ba{H} = \ba{H_0}$. Now $GK=FHR=F_0H_0K_{[1]}R$  so $ G=F_0 H_0R = (F_0 \times H_0) \rtimes R$.  Since $H_0$ is characteristic in $H$, it is normal in $J$. Recall that $K\cap G=N\leq H_0$, hence $K\cap F_0R=1$. As  $ \ba{J}\cap \ba{F_0R}=1$, this implies that $J\cap  F_0R=1$. Since $H_0\leq J$, we have $J \cap G = H_0(J\cap  F_0R)=H_0$, which is (iii). Note that $H_0/N\cong\ba{H}$. Since $\ba{H}$ contains its centraliser in $\ba{J}$, we have $\C_J(H_0) \leqslant HK=H_0K_{[1]}$. As $N\leq H_0$ and $N$ is self-centralising in $K$, we have $\C_J(H_0)\leq H_0$, which is (iv).

This concludes the proof.
\end{proof}

We now build on the previous result and give some information about the structure of the connection set.

\begin{thm}\label{cor:graphs}
Let $G$ be a Sylow cyclic group whose order is not divisible by four, let $\Gamma=\Cay(G,S)$ be a connected non-CCA graph  and let $A=\Aut_c(\Gamma)$. Using the notation of Theorem~\ref{thm:square-free}, write $A=(T\times J) \rtimes R$ and $G=(F\times H)\rtimes R$.  Let $r$ be the generator of $R$, let $Y=S\setminus (F\cup (H\rtimes R))$ and let
$$\Gamma'=\Cay(F\rtimes R,(F\cap S)\cup \{r\}\cup \{s^2\colon s \in Y\}).$$
Then
\begin{enumerate}
\item $\Gamma'$ is connected and non-CCA,
\item $Y\subseteq \{fz\colon f \in F,z \in Hr,  |f|=3, |z|=2\}$, and
\item if $Y\neq\emptyset$, then $|R|=2$, and $T$ commutes with $R$.
\end{enumerate}
\end{thm}

\begin{proof}
Since $\Gamma$ is non-CCA, $G$ is not normal in $A$. This yields the conclusion of Theorem~\ref{thm:square-free}. As in the proof of that theorem,  for $g\in G$, we write $[g]$ for the vertex of $\Cay(G,S)$ corresponding to $g$ and, for $X\subseteq G$, we write $[X]$ for  $\{[x]\colon x\in X\}$. 

Let $P$ be a Sylow $2$-subgroup of $A$ containing $R$. Up to relabelling, we may assume that $A_{[1]}\leq P$ and thus $P=A_{[1]}R$. It follows that $[1]^{PH}=[H\rtimes R]$ is a block for $A$. As $T$ is normal in $A$, its orbits are also blocks. One such block is $[1]^T=[F]$.  As $[F]\cap [H\rtimes R] =[1]$, we find that the two block systems induced by $[F]$ and by $[H\rtimes R]$ are transverse.

The action of $T$ on $[F]$ is equivalent to the action of $\PSL(2,7)$ by conjugation on its $21$ Sylow $2$-subgroups. In particular, if $f\in F$ and $T_{[1]}=T_{[f]}$, then $f=1$. This observation, together with the previous paragraph, yields that the set of fixed points of $T_{[1]}$ is exactly $[H\rtimes R]$. 

We first show (2) and (3). Let $s\in Y$. Note that $[Fs]$ is the orbit of $T$ that contains $s$. Since $s\notin H\rtimes R$, $[s]$ is not fixed by $T_{[1]}$. Since $T$ is colour-preserving, we have that $[s]\neq [s^{-1}]\in [Fs]$. It follows that  $1\neq s^2\in F$. In particular, $|s^2|\in \{3,7\}$. Since $A$ is colour-preserving, the cycles coloured $\{s,s^{-1}\}$ form a block system for $A$. This means that $[\langle s^2\rangle]$ is also a block for $A$, contained in the block $[F]$. Now, $\PSL(2,7)$ on its action on $21$ points does not admit blocks of size $7$, therefore $|s^2|=3$.  Since $s \not\in F$ this implies $|s|=6$. Notice also that $[\{1,s^3\}]$ is a block of $A$. Thus, $[s^3]$ is a fixed point of $T_{[1]}$, so $[s^3] \in [H \rtimes R]$. Since $|s^3|=2$ but $|H|$ is odd, $s^3 \not\in H$ hence $|R|=2$  and $s^3 \in Hr$. Since $H$ and $F$ centralise each other and $s^3r\in H$ and $s^2\in F$,  it follows that $s^2$ commutes with $r$. Note that $[1]^P=[R]$ is a block for $A$.  Now, $[\langle s^2\rangle]$ is also a block for $A$, being a set of vertices of even distance contained in one of the monochromatic hexagons coloured $\{s,s^{-1}\}$. It follows that $[\langle s^2,r\rangle]$ is also a  block for $A$,  of size $6$ and contained in the block  $[1]^{TP}=[F\rtimes R]$.
Note that $\PGL(2,7)$ does not have blocks of size $6$ in its transitive action on $42$ points. It follows that $T\rtimes R\not\cong \PGL(2,7)$, and hence $T$ commutes with $R$.
Writing $f=s^4$ and $z=s^3$ concludes the proof of (2) and (3).

Let $\pi: G\mapsto F\rtimes R$ be the natural projection and let $s\in Y$. By the previous paragraph, we have $s^{-1}=s^2s^3=s^2hr$, where $s^2\in F$ and $h\in H$. It follows that $\pi(s^{-1})=s^{2}r$.  As $S$ is inverse-closed, we have $\pi(Y)=\{s^2r:s\in Y\}$. Since $\langle S\rangle=G$, we  have $F\rtimes R=\langle \pi(S)\rangle\leq \langle F\cap S,r, s^2r \colon s\in Y\rangle=\langle F\cap S,r, s^2 \colon s\in Y\rangle$ and thus $\Gamma'$ is connected. Note that $[1]^{T\rtimes R}=[F\rtimes R]$ hence  $T\rtimes R\leq \Aut_c(\Gamma')$. Since $F\rtimes R$ is not normal in $T\rtimes R$, $\Gamma'$ is not CCA.
\end{proof}

The following result is, in some sense, a converse to Corollary~\ref{cor:graphs}. 

\begin{prop}
Let $G$ be a Sylow cyclic group whose order is not divisible by four such that $G=(F \times H) \rtimes R$ where $F \cong \F_{21}$, $R$ is a Sylow $2$-subgroup of $G$, and $F$ and $H$ are normal in $G$. Let $r$ be the generator of $R$, let $S$ be a generating set for $G$, let $Y=S\setminus (F\cup (H\rtimes R))$, let $S'=(F\cap S)\cup \{r\}\cup \{s^2\colon s \in Y\}$, and let
$$\Gamma'=\Cay(F\rtimes R,S').$$
If 
\begin{enumerate}
\item $\Gamma'$ is connected and non-CCA,
\item $Y\subseteq \{fz\colon f \in F,z \in Hr,  |f|=3, |z|=2\}$, and
\item if $Y\neq\emptyset$, then $|R|=2$, and $F$ commutes with $R$,
\end{enumerate}
then $\Cay(G,S)$ is connected and non-CCA.
\end{prop}

\begin{proof} 
Since $S$ generates $G$, $\Cay(G,S)$ is connected. Since $\Gamma'$  is a connected and non-CCA Cayley graph on $F\rtimes R$, it follows  from Theorem~\ref{thm:square-free} that there exists a group $T\rtimes R$ of colour-preserving automorphisms of $\Gamma'$, with $F\leq T$ and $T \cong \PSL(2,7)$.

This yields an action of $T$ on $F\rtimes R$. We extend this action to  the vertex-set of $\Cay(G,S)$ in the following way: for $t\in T$ and $xh\in G$, with $x\in F\rtimes R$ and $h\in H$, let $(xh)^t=x^th$.

Notice that if $x\in F\rtimes R$, then, since $r \in S'$ is an involution and $T \rtimes R$ is colour-preserving on $\Gamma'$, for any $t \in T$ we have $(rx)^t=rx^t$.

Note that $F\leq T\cap G < T$. Since $T$ is simple, it follows that $T\cap G$ is not normal in $T$. We claim that $T$ is a colour-preserving group of automorphisms of $\Cay(G,S)$. By the previous comment, this will show that $\Cay(G,S)$ is non-CCA.

Let $t\in T$, let $v\in G$ and write $v=xh$ with $x\in F\rtimes R$ and $h\in H$. We will show that, for all $s\in S$, we have $(sv)^t=s^{\pm 1}v^t$.

Suppose first that $s\in S'$. (This includes the case when $s\in F$.) Since $T$ is colour-preserving on $\Gamma'$, we have $(sx)^t=s^{\pm 1}x^t$. Since $sx\in F\rtimes R$, we have $(sv)^t=(sxh)^t=(sx)^th=s^{\pm 1}x^th=s^{\pm 1}v^t$, as required.

Suppose next that $s \in H \rtimes R$. Write $s=h'r^i$ and $x=r^jf$, where $h'\in H$, $f\in F$ and $i,j\in \ZZ$. Let $h'' \in H$ be such that $r^{i+j}h''=h'r^{i+j}$. Then 
\begin{align*}
(sv)^t&=(h'r^{i+j}fh)^t=(r^{i+j}h''fh)^t=(r^{i+j}fh''h)^t=(r^{i+j}f)^th''h=r^{i+j}f^th''h\\
&=r^{i+j}h''f^th=h'r^{i+j}f^th=sr^jf^th=s(r^jf)^th=sx^th=sv^t,
\end{align*}
as desired.

Finally, suppose $s \in Y$. We can write $s=fz$ where $f \in F$, $z \in Hr$,  $|f|=3$ and $|z|=2$. By (3), we have $s^3=z$, $s^2=f^2$ and $|s|=6$.  Since $s^3 \in H \rtimes R$, the argument of the previous paragraph shows $(s^3v)^t=s^3v^t$. On the other hand, since $s^2\in S'$, we have $(s^3v)^t=(s^2(sv))^t=s^{\pm2} (sv)^t$. Combining these gives $(sv)^t=s^{3\pm 2}v^t=s^{\pm 1} v^t$, as desired.
\end{proof}

We view Theorem 5.2 as a  reduction of the CCA problem for groups of the kind appearing in its statement to the determination of non-CCA graphs on $\F_{21}$ and $\AGL(1,7)$.
It therefore becomes of significant interest to understand the structure of such graphs.

Let $x$ and $y$ be elements of order $7$ and $6$ in $\AGL(1,7)$, respectively,  and let $d=(y^3)^x$. Note that $\langle x,y^2\rangle=\F_{21}$. Let
$$S_{21}=\{y^{\pm 2}, (xy^2)^{\pm1}\},~~S_{42,1}=\{y^{\pm 2},d\} \textrm{ and } S_{42,2}=\{y^{\pm 2},(y^{\pm 2})^{d},d\}.$$

Note that $\Cay(\F_{21},S_{21})$ is isomorphic to the line graph of the Heawood graph (see Example~\ref{eg:f21}), while $\Cay(\AGL(1,7),S_{42,1})$ is isomorphic to the line graph of the subdivision of the Heawood graph (see Example~\ref{eg:42}). 

\begin{prop}\label{prop:non-CCA}\label{prop:non-CCA-F21xC2}\mbox{}
\begin{enumerate}
\item The graph $\Cay(\F_{21},S)$ is connected but not CCA if and only if $S$ is conjugate in $\AGL(1,7)$ to $S_{21}$. 
\item The graph $\Cay(\AGL(1,7),S)$ is connected but not CCA if and only if $S$ is conjugate  in $\AGL(1,7)$ to one of $S_{42,1}$ or $S_{42,2}$.
\item The graph $\Cay(\F_{21}\times \ZZ_2,S)$ is connected but not CCA if and only if $S$ is conjugate in 
$\AGL(1,7) \times \ZZ_2$
 to some inverse-closed subset of 
$$\{y^{\pm 2},(xy^2)^{\pm 1}, y^{\pm 2}r, (xy^2)^{\pm 1}r, r\}$$
that generates $\F_{21}\times \ZZ_2$, where $\ZZ_2=\langle r \rangle$. 
\end{enumerate}
\end{prop}
\begin{proof}
This was verified using \textsc{magma}~\cite{magma}. The proof of the first claim can also be found in~\cite[Proposition~2.5, Remark~2.6]{CCASquarefree}. 
\end{proof}

\begin{rem}
It can be checked that Proposition~\ref{prop:non-CCA-F21xC2} (3) yields eleven generating sets for $\F_{21}\times \ZZ_2$, up to conjugacy in $\AGL(1,7) \times \ZZ_2$.
\end{rem}

\noindent\textsc{Acknowledgements.}
We would like to thank Gordon Royle for his help with some of the computations. We would also like to thank the anonymous referee for a number of helpful suggestions.


\begin{thebibliography}{99}
\bibitem{magma} W.~Bosma, J.~Cannon and C.~Playoust, The \texttt{Magma} algebra system. I: The user language, \emph{J. Symbolic Comput.} \textbf{24} (1997), 235--265. 

\bibitem{CompleteCCA} D.~Byrne, M.~Donner and T.~Sibley, Groups of graphs of groups, \emph{Beitr. Algebra Geom.} \textbf{54} (2013), 323--332.

\bibitem{CCASquarefree} E.~Dobson, A.~Hujdurovi\'{c}, K.~Kutnar, J.~Morris, On color-preserving automorphisms of Cayley graphs of odd square-free order, \emph{J. Algebraic Combin.} \textbf{44} (2017), 407--422.


\bibitem{VTCCA} E.~Dobson, A.~Hujdurovi\'{c}, K.~Kutnar, and J.~Morris, Vertex-transitive digraphs with extra automorphisms that preserve the natural arc-colouring, \emph{Australas. J. Combin.} \textbf{67} (2017), 88--100.


\bibitem{FirstCCA} A.~Hujdurovi\'{c}, K.~Kutnar, D.~W.~Morris, J.~Morris, On colour-preserving automorphisms of Cayley graphs, \emph{Ars Math. Contemp.} \textbf{11} (2016), 189--213.

\bibitem{CirculantCCA} J.~Morris, Automorphisms of circulants that respect partitions, \emph{Contrib. Discrete Math.} \textbf{11} (2016), 1--6.

\bibitem{Stellmacher} H.~Kurzweil, B.~Stellmacher, \emph{The theory of finite groups. An introduction.} Translated from the 1998 German original. Universitext. Springer-Verlag, New York, 2004.

\bibitem{2Complement} C.~Martinez-Perez and W.~Willems, The trivial intersection problem for characters of principal indecomposable modules, \emph{Adv. Math.} \textbf{222} (2009), 1197--1219. 

\end{thebibliography}
\end{document}